\documentclass[12pt,notitlepage]{report}
\title{One sided conformal collars and the reflection principle}
\author{V. Liontou and V. Nestoridis}

\usepackage[english]{babel}
\usepackage[utf8]{inputenc}
\usepackage{titlesec}
\usepackage{amsthm}
\usepackage{amsmath}
\usepackage{amsfonts}
\usepackage{mathrsfs}
\usepackage{enumerate}
\usepackage{tikz}

\titleformat*{\section}{\normalsize\bfseries}
\titleformat*{\subsection}{\Large\bfseries}

\titleformat*{\section}{\normalsize\bfseries}
\titleformat*{\subsection}{\Large\bfseries}
\newtheorem{theorem}{Theorem}[chapter]
\numberwithin{theorem}{section}
\newtheorem{defi}[theorem]{Definition}
\newtheorem{lemma}[theorem]{Lemma}
 \newtheorem{corollary}{Corollary}[theorem]

\theoremstyle{definition}

\newtheorem{remark}[theorem]{Remark}

\addto\captionsenglish{}

\begin{document}
\maketitle

\begin{abstract}
If a Jordan curve $\sigma$ has a one-sided conformal collar with "good" properties, then, using the Reflection principle, we show that any other conformal collar of $\sigma$ from the same side has the same "good" properties. A particular use of this fact concerns analytic Jordan curves, but in general the Jordan arcs we consider do not have to be analytic. We show that if an one-sided conformal collar bounded by $\sigma$ is of class $A^{p}$,  then any other collar bounded by $\sigma$ and from the same side of $\sigma$  is of class $A^{p}$. 

\end{abstract}
A.M.S classification no: primary 30C99, 30E25 secondary30H10 
\section{Introduction}
In this paper, for two conformal mappings $\Phi$ and $F$,  we write $ F= \Phi \circ (\Phi^{-1}\circ F)$ whenever this has a meaning, as explained later. Using our assumptions, the function $\Phi^{-1}\circ F$ maps a circular arc $\tau$ on another circular arc. Thus, by Reflection $h=\Phi^{-1}\circ F$ extends to an injective holomorphic function on a neighbourhood $V$ of a part of $\tau$. Thus, $h^\prime$ extends continuously on $\tau$ and does not vanish at any point of $\tau$. Since $F=\Phi\circ h$  and $F^\prime= (\Phi^\prime \circ h) h^{\prime}$, if $\Phi$ or $\Phi^\prime$ has some good properties we can transfer them to $F$ and $F^\prime$.
The first application of this fact is described below.\\
Let $J$ be a Jordan curve in $\mathbb{C}$, that is a homeomorphic image of the unit circle $T$ defined by $\gamma :T\rightarrow J$. The Jordan curve $J$ is called analytic if there is a two sided conformal collar around it . That is, if there exists an annulus  $D(0,r_1,r_2)=\{z\in \mathbb{C}: r_1<|z|<r_2\}, 0<r_1<1<r_2$ and an injective holomorphic mapping $\phi: D(0,r_1,r_2)\rightarrow\mathbb{C}$ such that $\phi(e^{i\theta})=\gamma(e^{i\theta})$ for all $e^{i\theta}\in T$, $\theta \in \mathbb{R}$. A consequence of the existence of the above two sided conformal collar is that, if $\Omega$ is the Jordan region bounded by $J$ and $\Phi:D\rightarrow\Omega$ is a Riemann mapping from the open disk $D$ onto $\Omega$, then $\Phi$ has a conformal extension on a larger disc $D(0,r)=\{z\in\mathbb{C}: |z|<r\},r>1$. In particular, the function $\Phi$ extends continuously on the unit circle $T$ and the same holds for the derivative $\Phi^\prime$ and $\Phi^\prime(e^{i\theta})\neq 0$ for all $e^{i\theta}\in T$.
\\ In the present paper we investigate which conclusions still hold if we do not assume the existence of a two-sided conformal collar around the Jordan curve $J$ but we assume the existence of a one sided conformal collar bounded by $J$.
\\ Thus, we assume the existence of an injective holomorphic function $\phi: D(0,r,1)\rightarrow \mathbb{C}$, $ D(0,r,1)=\{z\in \mathbb{C}: r<|z|<1\}$ which extends continuously on $\overline {D(0,r,1)}$ and let us call $\phi$ the extension as well. We also assume that $\phi $ is one to one on $D(0,r,1)\cup T$ where $T=\{z\in \mathbb{C}:|z|=1\}$ and that $\phi_{|_T}:T\rightarrow J$ is a homeomorphism. Futhermore,we assume that $\phi^\prime$ has a continuous extension on $\overline{D(0,r,1)}$ and that $\phi^\prime(e^{i\theta})\neq 0$ for all $\theta \in \mathbb{R}$. Then we show that the Riemann map $\Phi :D\rightarrow \Omega $, onto the Jordan region $\Omega$ bounded by $J$ has similar properties; that is, $\Phi$ and $\Phi^\prime$ extend continuously on $\overline{D}$ and $\Phi^ \prime (e^{i\theta})\neq 0$ for all $\theta\in \mathbb{R}$. Certainly ,the fact that $\Phi$ extends continuously on $\overline{D}$ is a consequence of the Ozgood-Caratheodory theorem \cite{[6]}. The rest is a consequence of the Reflection Principle (\cite{[1]}). We also can extend the above fact for every finite number $ p+1$ of derivatives, provided that the one-sided conformal collar is of class $A^{p}$.(\cite{[3]},\cite{[5]}) Furthermore, instead of considering the unit circle we can replace it by one-sided free boundary arcs (\cite{[1]})which are analytic arcs and obtain similar results (see Th.2.2 below)
\\
 We also mention that our result relates to the considerations of (\cite{[2]})
\\
\break

Finally, a second application of the method of the present paper can be found in \cite{[7]}, see proof of Th. 4.2 where it is proven the following. If $\Omega$ is a domain in $\mathbb{C}$, whose boundary contains an analytic arc $J$ and the boundary values of a holomorphic function $f \in H(\Omega)$ define a function in $C^{p}(J)$, then the derivatives $f^{(l)}(z), z\in \Omega ,l=0,...,p $ extend continuously on J. 
\section{One sided conformal collar}
We start with the following definition.
\begin{defi}
A complex function $\phi(t)$ of a real variable $t$ defined on an interval $a<t<b$, defines an analytic curve if, for every $t_0$ in the  $(a,b)$,the derivative $\phi^\prime(t_0)$ exists and it is non zero and $\phi$ has a representation  $$\phi(t)=\sum_{n=0}^{\infty}\alpha_n(t_0)(t-t_0)^{n}, \alpha_n(t_0)\in \mathbb{C},$$ where the power series converges in some interval $(t_0-p,t_0+p), p=p(t_0) > 0$, and coincides with $\phi$ in this interval.
\end{defi}

 But if this is so, then from Abel's lemma, the series is also convergent for complex values of $t$, as long as $|t-t_0|<p$ and represents an analytic function on that disk. In overlapping disks the functions are the same, because they coincide on a segment of the real axis. Thus the function $\phi(t)$ can be extended as an analytic and injective function in a neighbourhood of $(a,b)$ in $\mathbb{C}$ .\cite{[1]}

\begin{theorem}
Let $W_1$ and $W_2$ be regions and $J_1$and $J_2$ be analytic arcs of the boundaries of the regions $W_1, W_2$ respectively. We also assume that there are neighbourhoods $U_1, U_2$ of $J_1, J_2$ respectively for which $U_1\cap W_1$ and $U_2\cap W_2$ are connected regions. For the analytic arcs $J_1$ and $J_2$ we assume that:
\begin{enumerate}
\item For every point of $J_1$ or $J_2$ there exists a neighbourhood whose intersection with the whole boundary of $W_1$ or $W_2$ respectively is the same with its intersection with $J_1$ or $J_2$.
\item For every point z in $J_i, i=1,2$ and for every disk $D(z,\epsilon)$ such that $D(z,\epsilon)\cap\partial W_i = D(z,\epsilon)\cap J_i$, the arc $J_i$ separates $D(z,\epsilon)$ in two subsets, one completely inside $W_i$ and the other completely outside $W_i$.
\item The arcs $J_i, i=1,2$ are simple analytic and the injective maps $\phi_1, \phi_2$ with $J_1 := \phi_1 :(a_1,b_1)\rightarrow \mathbb{C}$ and $J_2:= \phi_2:(a_2,b_2)\rightarrow \mathbb{C}$ satisfy that $\phi_1^{\prime}$ and $\phi_2^{\prime}$ exist and are non-zero for every point of the segments $(a_1,b_1)$ and $(a_2,b_2)$ respectively
\end{enumerate}
We also assume that a function $\phi: W_1\cup J_1\rightarrow \mathbb{C}$ is injective, continuous on $W_1\cup J_1$ and holomorphic in $W_1$ and we denote  $G:=\phi(W_1)$. We suppose that $\phi^{\prime}$ can be extended in $W_1\cup J_1$ continuously. Let $\gamma=\phi(J_1)$ and let a function $F:W_2\rightarrow \mathbb{C}$ be continuous on $W_2\cup J_2$ and holomorphic on $W_2$. Suppose that for all z in $\phi(J_1)$, there exists a $\epsilon > 0$ so that $D(z,\epsilon)\cap F(W_2)\subseteq G$ and if a sequence $(z_n)_{n\in \mathbb{N}}, z_n\in W_2$ for all $n\in \mathbb{N}$ accumulates only in $J_2$, then $F(z_n)$ accumulates only in $\gamma$ and $F(J_2)\subseteq \gamma = \phi(J_1)$.\\
Then: 

\begin{enumerate}[(a)]

\item The function $F^\prime$, can be continuously extended in $W_2\cup J_2$.
\item If $\phi^\prime(z)\neq 0 $  for all $z$ in $J_1$ and $F$ is injective in $W_2$ then $F^\prime(\zeta)\neq 0$  for all $\zeta$ in $J_2$ 
\end{enumerate}
\end{theorem}
\begin{proof}
\noindent
First we prove statement (a).
\\
 Let $z\in J_2$. There exists a $ \delta_z >0 $ such that $ D(z,\delta_z)\cap (\partial W_2 \setminus J_2)=\emptyset$. Therefore, the function  $$ h:=\phi^{-1}\circ F:D(z,\delta_z)\cap W_2\rightarrow W_1$$ is well defined. The function $h$ is holomorphic, since F is holomorphic in $W_2$. The function $$\phi^{-1}:\phi(W_1)\rightarrow W_1$$ is holomorphic also holomorphic in $\phi(W_1)$ and if $\delta_z > 0$ is small enough, then $F\big(D(z,\delta_z)\cap W_2\big)\subseteq \phi(W_1)$, for every $z\in \gamma$.\\
 We assume that $\{z_n, n \in\mathbb{N}\} $ is a sequence of points in $W_2$. We are going to prove that if all the accumulation points of $\{z_n\}_{n\in \mathbb{N}}$ belong in $J_2$, then all the accumulation points of sequence $b_n:=h(z_n),n\in\mathbb{N}$ belong in $J_2$. From our hypothesis it follows that $F(z_n),n\in \mathbb{N}$ accumulates in $\gamma$. It suffices to prove that \{$\phi^{-1}\circ F(z_n), n\in \mathbb{N}\}$ accumulates in $J_1$. If we suppose that there is an accumulation point $w$ of the sequence $b_n:=\phi^{-1}\circ F(z_n)$ that does not belong in $J_1$, then ,there exists a subsequence $\{b_{k_n}, n\in \mathbb{N}\}$ that converges to $w\in W_1$. Thus, $\phi( b_{k_n})$ defines a subsequence of $F(z_n)$ in $\phi(W_1)$. Since $\phi$ is continuous on $W_1\cup J_1$,then $\phi(z_{k_n})$ converges to $\phi(w)$. Since all accumulation points of $F(z_n)$ belong in $\gamma$, therefore $\phi(w)$ belongs in $\gamma$. Thus, there exists an $z_1 \in J_1$ such as $\phi(z_1)=\phi(w)$. This is impossible since $z_1$ is on the boundary of $W_1$ and $w$ inside of $W_1$ and $\Phi$ is injective on $W_1\cup J_1$. 
  \\ Therefore if $(z_n)_n$ accumulates on $J_2$,  $h(z_n)$ accumulates  on $J_1$.\noindent
  \\
  \\
Finally, $h$ is holomorphic on $D(z ,\delta_z)\cap W_2$ and continuous on $\overline {D(z ,\delta_z)\cap W_2}$. According to the Reflection Principle for analytic arcs \cite{[1]}, $h $ is holomorphic  in $D(z,\delta_z')$for some $\delta_z', 0<\delta_z'\leq \delta_z$. Therefore $F=\phi\circ\phi^{-1}\circ F$ is continuous in $J_2$ since $\phi$ can be continuously extended in $J_1$.\\
Considering  $F$ restricted in ${ D(z,\delta_z)\cap W_2}$ it is analytic and $$F^\prime(z)=[\phi\circ(\phi^{-1}\circ F)]^\prime(z)=\phi^\prime\circ(\phi{-1}\circ F)(z)\cdot(\phi^{-1}\circ F)^\prime(z)$$ $$\forall z\in D(z,\delta_z)\cap W_2$$
Since $\phi\prime$ is continuous in $J_1$  and from the Reflection Principle $(\phi^{-1}\circ F)^\prime$ is holomorphic in $D(z,\delta_z)$, therefore $F^\prime$ is continuous on $\overline{D(z,\delta_z)\cap W_2}$.
\\

The above statements are valid for all $z\in 
 J_2$. From analytic continuation the extensions of $h \in D(z_i,\delta_{z_i}')\cap D(z_j,\delta_{z_j}')\cap W_2$ match, thus they also match in $D(z_i,\delta_{z_i}')\cap D(z_j,\delta_{z_j}')$, for any $z_1, z_2 \in J_2$\\
 Therefore $F^\prime$ can be continuously extended in $J_2$\\
 \\
 Next we prove statement (b):
\\
\break
 We have already that $F^\prime=[\phi\circ(\phi^{-1}\circ F)]^\prime=\phi^\prime\circ(\phi^{-1}\circ F)\cdot(\phi^{-1}\circ F)^\prime$\\
Let $z\in J_2$. Since for all $ z_n$ accumulating in $J_2$, the sequence $F(z_n)$ accumulates in $\gamma$, it follows that $\phi^{-1}\circ F(z) \in J_1$. Assuming that $\phi^\prime \neq 0$, for all z in $J_1$, it follows that $\phi^\prime\circ(\phi^{-1}\circ F)(z)\neq 0$ for all $z$ in $J_2$.\\
From the Reflection Principle, the function $h=\phi^{-1}\circ F$ can be extended in an  neighbourhood of $J_2$  and the extension is injective on it; thus $(\phi^{-1}\circ F)^\prime\neq 0$, for all $z$ in $J_2 $.Therefore $ F^\prime (z)\neq 0$ for all $z\in J_2$. This completes the proof.
\\

\end{proof}
\begin{remark}
Let suppose that there exists $z$ in $J_1$ such that $\phi^\prime(z)\neq 0$ and not for all points in $J_1$. Therefore, there exists $\zeta$ in $J_2$ such that $\phi^{-1}\circ F(\zeta)=z$ in $J_1$ and $F(\zeta)\neq 0$ and reversely.

\end{remark}
\begin{lemma}
Let $W\subseteq \mathbb{C}$ and $\Omega\subseteq\mathbb{C}$ open sets. Let $\phi: W\rightarrow \mathbb{C}$ and $h:\Omega\rightarrow W$ holomorphic functions. We assume that $P$ is a finite linear combination of products of the functions $\phi^{(n)}\circ h, \phi^{(n-1)}\circ h,..., \phi\circ h$ and $h^{(n)},..., h$, for some $n$ in $\mathbb{N}$. Then $P$ is holomorphic on $\Omega$ and its derivative is a polynomial of the functions $\phi^{(n)}\circ h, \phi^{(n-1)}\circ h ,...,\phi\circ h$ and $h^{(n)},..., h$ and $\phi^{(n+1)}\circ h, h^{(n+1)}$
\end{lemma}
\begin{proof}
We will use induction to prove this lemma. We start from $$P(z)=(\phi\circ h)^{k_1} h^{k_2} (z)$$, then
$$P^\prime(z)=k_1(\phi^\prime\circ h)h^\prime h^{k_2}+ k_2(\phi\circ h^{k_1})h^{k_2-1}h^\prime$$
For the inductive step we assume $$P=c[ \phi^{(n-1)}\circ h ^{k_{n^2}}\dots \phi\circ h^{k_{n+1}}{h^{(n)}}^{k_n}\dots h^{k_1}], c \in \mathbb{C}$$
Without loss of the generality, we also assume that $k_{2n}=\dots =k_1=c=1$, hence,
$$P^\prime=(\phi^{(n)}\circ h)h^\prime[\phi^{(n-1)}\circ h)\dots (\phi^{-1}\circ F)]$$ $$ + (\phi^{(n-1)}\circ h)^2 h^\prime (\phi^{(n-2)}\circ h\dots h)$$$$ +\dots + [\phi^\prime \circ h]^{(n+1)}([\phi^{(n)}\circ h]\dots h)$$$$+ h^{(n+1)}([\phi^{(n)}\circ h])+\dots +h([\phi^{(n)}\circ h]\dots h^\prime) $$
Therefore $P^\prime$ is a polynomial of $(\phi^{(n)}\circ h)^{(k)}, h^{(k)}, k=0,..., n+1$ 
\end{proof}
\begin{corollary}
 According to assumptions and notation of Theorem 2.2, if $\phi$ is an analytic function in $W_1$ and continuous on $J_1$ and we also assume that $\phi^{(k)}$ can  be continuously extended on $J_1$ for $k=0,1,...,p$, where p is a natural number, then $F^{(k)}$ can be continuously extended on $W_2\cup J_2$ for $k=0,1,...,p$.\\ 
 \end{corollary}
  
 This holds because $F^{(k)}$ is a polynomial of functions that can be continuously extended on $J_2$. This follows combining Theorem 2.2 with lemma 2.4.\\
(for more details see \cite{[3]})\\
\break

The Theorem 2.2 implies, also, the following corollary. 

\begin{corollary}
We suppose that $D(0,r,1) ,0<r<1$ is an open annulus and $\phi: \overline{D(0,r,1)}\rightarrow\mathbb{C}$ is continuous, injective on $\overline{D(0,r,1)}$ and holomorphic on $D(0,r,1)$. We also assume that $\phi^\prime$ has a continuous extension on $\overline {D(0,r,1)}$. Furthermore, we set $\gamma(e^{i\theta}):= \phi(e^{i\theta})$ and $W$ is the interior of the Jordan curve $\gamma$.
\\We, finally, assume that $\phi(D(0,r,1))\subseteq W$ and the function 
$$F:D(0,1)\rightarrow W$$ 
is a conformal mapping of $D:= D(0,1)$ onto W. It is known that according to the Ozgood-Caratheodory theorem ,$F$ is exteded to a homeomorphism $F: \overline{D(0,1)}\rightarrow \overline{W}$. From the Theorem 2.2, it follows that\\
\begin{enumerate}
\item The function $F^{\prime}$ extends continuously on $\overline D$
\item If $\phi^\prime(z)\neq 0$, for some $z$ with $|z|=1$ then\\  $F^\prime(\zeta)\neq 0$ , for $\zeta$ such that $(\phi^{-1}\circ F)(\zeta)=z$
\end{enumerate}
\end{corollary}
More generally, if a conformal collar of a Jordan curve has some nice properties, then the same holds for any other conformal collar of the same curve from the same side.

National and Kapodistrian University of Athens,\\ Department of Mathematics \\15784 \\Panepistemiopolis\\ Athens GREECE
\\
e-mail: lvda20@hotmail.com\\
e-mail: vnestor@math.uoa.gr


\begin{thebibliography}{3}
\bibitem{[1]}Ahlfors, Complex analysis,Second edition ,McGraw-Hill,New York,1966
\bibitem{[2]}Steven R.Bell and Steven G Krantz, Smoothness to the boundary of conformal maps, Rocky Mountain Journal of mathematics, Volume 17,1987, Number 1 ,23-40
\bibitem{[3]}E.Bolkas,V.Nestoridis,C. Panagiotis and M.PapadimitrakisOne sided extendability and p-continuous analytic capacities,arxiv : 1606.05443
\bibitem{[4]}Gauthier,P.M., Nestoridis V., Conformal extensions of functions defined on arbitrary subsets of Riemann Surfaces. Arch. Math. ( Basel ) 104 (2015 )no 1  , 61-67.
\bibitem{[5]}Georgakopoulos N. ,Extensions of the Laurent Decomposiotion and the spaces $A^p(\Omega)$, arxiv:1605.08289 
\bibitem{[6]}Koosis P. ,An introduction to $H_p$ spaces, Cambridge University Press, 1998
\bibitem{[7]}V.Mastrantonis, Relations of the spaces $A^{p}(\Omega)$ and $C^{p}(\partial\Omega)$, arxiv:1611.02971 
\end{thebibliography}
\end{document}